\documentclass[10pt,reqno]{amsart}

\usepackage{mathtools}

\usepackage{lmodern}
\usepackage{mathabx}

\usepackage[T1]{fontenc}

\usepackage{amsmath,amsfonts,amsbsy,amsgen,amscd,mathrsfs,amssymb,amsthm}
\usepackage{enumerate}
\usepackage{bm}

\usepackage{stmaryrd}
\SetSymbolFont{stmry}{bold}{U}{stmry}{m}{n} 

\usepackage[usenames,dvipsnames]{xcolor}
\usepackage[colorlinks=true,citecolor=blue,linkcolor=blue]{hyperref}

\usepackage{tikz}

\newtheorem{thm}{Theorem}[section]
\newtheorem{lem}[thm]{Lemma}

\newtheorem{prop}[thm]{Proposition}
\newtheorem{cor}[thm]{Corollary}

\theoremstyle{definition}

\newtheorem{defn}[thm]{Definition}
\newtheorem*{defn*}{Definition}

\newtheorem*{claim}{Claim}

\theoremstyle{remark}

\newtheorem{rem}[thm]{Remark}
\newtheorem*{rem*}{Remark}

\numberwithin{equation}{section} 
\numberwithin{figure}{section}
\numberwithin{table}{section}

\makeatletter
\renewcommand\subsubsection{\@startsection{subsubsection}{3}%
  \z@{.5\linespacing\@plus.7\linespacing}{-.5em}%
  {\normalfont\bfseries}}
\makeatother

\newcommand{\Vol}{\mathrm{Vol}}
\newcommand{\V}{\mathsf{V}}
\newcommand{\D}{\mathsf{D}}

\newcommand{\supp}{\mathop{\mathrm{supp}}}

\newcommand{\proj}{\bm{\mathsf{P}}}

\begin{document}

\title[Log-Brunn-Minkowski and Zonoids]{The local logarithmic Brunn-Minkowski 
inequality for zonoids}

\author{Ramon van Handel}
\address{Fine Hall 207, Princeton University, Princeton, NJ 
08544, USA}

\begin{abstract}
The aim of this note is to show that the local form of the 
logarithmic Brunn-Minkowski conjecture holds for zonoids. The proof uses
a variant of the Bochner method due to Shenfeld and the author.
\end{abstract}



\maketitle

\thispagestyle{empty}

\section{Introduction}
\label{sec:intro}

\subsection{}

The classical Brunn-Minkowski inequality states that
\begin{equation}
\label{eq:bm}
	\Vol((1-t)K+tL)^{1/n} \ge
	(1-t)\,\Vol(K)^{1/n}+t\,\Vol(L)^{1/n}
\end{equation}
for all $t\in[0,1]$ and convex bodies $K,L$ in $\mathbb{R}^n$, where
$$
	aK+bL:=\{ax+by:x\in K,y\in L\}
$$
denotes Minkowski addition. Its importance, both to convexity and to other 
areas of mathematics, can hardly be overstated; cf.\ \cite{Gar02}.
As is well known, \eqref{eq:bm} is equivalent to the apparently weaker 
inequality
\begin{equation}
\label{eq:bmgm}
	\Vol((1-t)K+tL) \ge
	\Vol(K)^{1-t}\,\Vol(L)^{t}
\end{equation}
where the arithmetic mean on the right-hand side has been replaced by the 
geometric mean. Clearly \eqref{eq:bm} implies \eqref{eq:bmgm}, as 
the geometric mean is smaller than the arithmetic mean; the converse 
implication follows by rescaling $K,L$ \cite[\S 4]{Gar02}.

As part of their study of the Minkowski problem for cone volume measures, 
B\"or\"oczky, Lutwak, Yang and Zhang \cite{BLYZ12} asked whether one could
replace also the ``arithmetic mean'' $(1-t)K+tL$ on the left-hand side of 
the Brunn-Minkowski inequality by a certain kind of ``geometric mean'': 
that is, whether
\begin{equation}
\label{eq:logbm}
	\Vol(K^{1-t}L^t)\mathop{\stackrel{?}{\ge}} 
	\Vol(K)^{1-t}\,\Vol(L)^t,
\end{equation}
where the meaning of $K^{1-t}L^t$ must be carefully defined (see
\eqref{eq:logplus} below). As the 
geometric mean is smaller than the arithmetic mean, this would yield an 
improvement of the classical Brunn-Minkowski inequality.
While such an improved inequality 
turns out to be false for general convex bodies, it was conjectured in 
\cite{BLYZ12} that such an improved inequality holds whenever $K,L$ are 
symmetric convex bodies (that is, $K=-K$ and $L=-L$), which they proved to 
be true in dimension 2. In higher dimensions, this \emph{logarithmic 
Brunn-Minkowski conjecture} remains open.

\subsection{}

It is readily seen that the Brunn-Minkowski inequality \eqref{eq:bm}
and the logarithmic Brunn-Minkowski conjecture \eqref{eq:logbm} 
are equivalent to concavity of the functions
$$
	\varphi: t\mapsto \Vol((1-t)K+tL)^{1/n}\qquad
	\mbox{and}\qquad
	\psi:t\mapsto \log \Vol(K^{1-t}L^t)
$$
for all convex bodies $K,L$ and symmetric convex bodies $K,L$
in $\mathbb{R}^n$, respectively. We can therefore obtain equivalent 
formulations of \eqref{eq:bm} and \eqref{eq:logbm} by considering the 
first- and second-order conditions for concavity of $\varphi$ and $\psi$.

In order to formulate the resulting inequalities, we must first recall 
some additional notions (we refer to \cite{Sch14} for a 
detailed treatment). It was shown by Minkowski 
that the volume of convex bodies is a polynomial in the sense that for any
convex bodies $K_1,\ldots,K_m$ in $\mathbb{R}^n$ and 
$\lambda_1,\ldots,\lambda_m>0$, we have
$$
        \Vol(\lambda_1K_1+\cdots+\lambda_m K_m)
        = \sum_{i_1,\ldots,i_n=1}^m \V(K_{i_1},\ldots,K_{i_n})\,
        \lambda_{i_1}\cdots\lambda_{i_n}.
$$
The coefficients $\V(K_1,\ldots,K_n)$, called \emph{mixed volumes}, are 
nonnegative, symmetric in their arguments, and homogeneous and
additive in each argument under Minkowski addition.
Moreover, mixed volumes admit the integral representation
\begin{equation}
\label{eq:mixvolarea}
	\V(K_1,\ldots,K_n) = \frac{1}{n}\int h_{K_1}
	dS_{K_2,\ldots,K_n},
\end{equation}
where the \emph{mixed area measure} $S_{K_2,\ldots,K_n}$ is a finite 
measure on $S^{n-1}$ and $h_K(x):=\sup_{z\in K}\langle z,x\rangle$ denotes 
the support function of a convex body $K$.

In view of the above definitions, it is now straightforward to obtain 
equivalent formulations of the Brunn-Minkowski inequality in terms of 
mixed volumes; see, e.g., \cite[pp.\ 381--382 and 406]{Sch14}. In the 
sequel, we denote by $\mathcal{K}^n$ ($\mathcal{K}_s^n$) the family of all 
(symmetric) convex bodies in $\mathbb{R}^n$ with nonempty interior.

\begin{lem}[Minkowski]
\label{lem:bmequiv}
The following are equivalent:
\begin{enumerate}[1.]
\itemsep\abovedisplayskip
\item For all $K,L\in\mathcal{K}^n$ and $t\in[0,1]$, the 
Brunn-Minkowski inequality \eqref{eq:bm} holds.
\item For all $K\in\mathcal{K}^n$, we have
\begin{equation}
\label{eq:minkfirst}
	\V(L,K,\ldots,K) \ge \Vol(L)^{1/n}\,\Vol(K)^{1-1/n}\quad
	\forall\,L\in\mathcal{K}^n.
\end{equation}
\item For all $K\in\mathcal{K}^n$, we have
\begin{equation}
\label{eq:minksecond}
	\V(L,K,\ldots,K)^2 \ge \V(L,L,K,\ldots,K)\,\Vol(K) \quad
	\forall\,L\in\mathcal{K}^n.
\end{equation}
\end{enumerate}
\end{lem}

\begin{proof}
If we apply 
\eqref{eq:minkfirst}--\eqref{eq:minksecond} with $K\leftarrow (1-s)K+sL$ 
and $L\leftarrow (1-r)K+rL$, then a simple computation shows that 
Minkowski's first inequality 
\eqref{eq:minkfirst} is nothing other than the first-order concavity 
condition $\varphi(r)\le \varphi(s)+\varphi'(s)(r-s)$, while Minkowski's 
second inequality \eqref{eq:minksecond} is the second-order condition 
$\varphi''(s)\le 0$.
\end{proof}

Before we state an analogous reformulation of \eqref{eq:logbm}, we must 
first  give a precise definition of $K^{1-t}L^t$. To motivate this 
definition, recall that the 
arithmetic mean of convex bodies is characterized by its support function
$h_{(1-t)K+tL}=(1-t)h_K+th_L$. We may therefore attempt to 
define $K^{1-t}L^t$ as the convex body whose support function is 
the geometric mean $h_K^{1-t}h_L^t$. However, the latter 
need not be the support function of any convex body. We therefore 
define $K^{1-t}L^t$ in general as the largest convex body whose support 
function is dominated by $h_K^{1-t}h_L^t$, that is,
\begin{equation}
\label{eq:logplus}
	K^{1-t}L^t := \{ z\in \mathbb{R}^n:
	\langle z,x\rangle \le h_K(x)^{1-t}h_L(x)^t\mbox{ for all }
	x\in\mathbb{R}^n \}.
\end{equation}
We can now formulate the following analogue of Lemma \ref{lem:bmequiv}.

\begin{thm}[\cite{BLYZ12,CLM17,KM17,CHLL20,Put21,Mil21}]
\label{thm:logequiv}
The following are equivalent:
\begin{enumerate}[1.]
\itemsep\abovedisplayskip
\item For all $K,L\in\mathcal{K}^n_s$ and $t\in[0,1]$, the 
log-Brunn-Minkowski inequality \eqref{eq:logbm} holds.
\item For all $K\in\mathcal{K}^n_s$, we have
\begin{equation}
\label{eq:logmink}
	\int h_K\log\bigg(\frac{h_L}{h_K}\bigg) dS_{K,\ldots,K}
	\ge 
	\Vol(K)\log\bigg(\frac{\Vol(L)}{\Vol(K)}\bigg)
	\quad
	\forall\, L\in\mathcal{K}^n_s.
\end{equation}
\item For all $K\in\mathcal{K}^n_s$, we have
\begin{equation}
\label{eq:locallogbm}
	\frac{\V(L,K,\ldots,K)^2}{\Vol(K)} \ge 
	\frac{n-1}{n}\,
	\V(L,L,K,\ldots,K)
	+
	\frac{1}{n^2} \int \frac{h_L^2}{h_K}\,dS_{K,\ldots,K}
	\quad
	\forall\,L\in\mathcal{K}^n_s.
\end{equation}
\end{enumerate}
\end{thm}

The difficulty in the proof of Theorem \ref{thm:logequiv} is that the map 
$t\mapsto K^{1-t}L^t$ can be nonsmooth: if it were the case that 
$h_{K^{1-t}L^t}=h_K^{1-t}h_L^t$ for all $t\in[0,1]$, the result would 
follow easily from the first- and second-order conditions for concavity of 
$\psi$. That the conclusion remains valid using the correct definition 
\eqref{eq:logplus} is a nontrivial fact that has been established through 
the combined efforts of several groups.

\begin{rem}
The notation \eqref{eq:logplus} is nonstandard: $K^{1-t}L^t$ is 
often denoted in the literature as $(1-t)K+_0 tL$, as it coincides with
the $q\to 0$ limit of $L^q$-Minkowski addition. As the latter notation is 
somewhat confusing (the geometric mean is not defined by the rescaled 
bodies $(1-t)K$ and $tL$), and as only geometric means are used in this 
paper, we have chosen a nonstandard but more suggestive notation.
\end{rem}

\subsection{}   

It was shown in \cite{BLYZ12} that the logarithmic Brunn-Minkowski 
conjecture holds in dimension $n=2$. In dimensions $n\ge 3$, however, the 
conjecture has been proved to date only under special symmetry 
assumptions: when $K,L$ are complex \cite{Rot14} or unconditional 
\cite{Sar15} bodies (see also \cite{BK20} for a generalization). In both 
cases the conjecture is established by replacing the geometric mean 
\eqref{eq:logplus} by a smaller set whose construction requires the 
special symmetries, which yields strictly stronger inequalities than are 
conjectured for general bodies.

Even for a fixed reference body $K$, the validity of the inequalities 
\eqref{eq:logmink} and \eqref{eq:locallogbm} for \emph{all} 
$L\in\mathcal{K}_s^n$ (i.e., in the absence of additional symmetries) 
appears to be unknown except in one very special family of 
examples: it follows from \cite{KM17,Mil21} that \eqref{eq:logmink} and 
\eqref{eq:locallogbm} hold when $K$ is the $\ell_p^n$-ball with $2\le 
p<\infty$ and sufficiently large $n$, as well as for affine images and 
sufficiently small perturbations of these bodies. Note, however, 
that the analysis of these examples shows that they satisfy even stronger 
inequalities that cannot hold for general bodies (local 
$L^q$-Brunn-Minkowski inequalities with 
$q=-\frac{1}{4}$ \cite[Theorem 10.4]{KM17}), so that they do not 
approach the extreme cases of the logarithmic Brunn-Minkowski 
conjecture.\footnote{%
	For one extreme case, the $\ell_\infty^n$-ball, the validity 
	of \eqref{eq:locallogbm} may be verified by an explicit 
	computation, see, e.g., \cite[Theorem 10.2]{KM17}. This does not 
	follow as a limiting case of the general result 
	\cite[Theorem 10.4]{KM17} on $\ell_p^n$-balls, however, as the 
	latter only holds for $n\ge n_0(p)\to\infty$ as 
	$p\to\infty$.
}

The aim of this note is to contribute some further evidence toward the 
validity of the logarithmic Brunn-Minkowski conjecture. Recall that a 
convex body $K\in\mathcal{K}_s^n$ is called a \emph{zonoid} if it is the 
limit of Minkowski sums of segments. The first main 
result of this note is the following theorem.

\begin{thm}
\label{thm:main}
Let $K\in\mathcal{K}_s^n$ be a zonoid. Then the local logarithmic 
Brunn-Minkowski inequality \eqref{eq:locallogbm}
holds for all $L\in\mathcal{K}_s^n$.
\end{thm}

Our second main result settles the equality cases of Theorem 
\ref{thm:main}.

\begin{defn*}
A vector $u\in S^{n-1}$ is called an \emph{$r$-extreme normal vector}
of a convex body $K$ if there do not exist linearly 
independent normal 
vectors $u_1,\ldots,u_{r+2}$ at a boundary point of $K$ such that 
$u=u_1+\cdots+u_{r+2}$.
\end{defn*}

\begin{thm}
\label{thm:equality}
Let $K\in\mathcal{K}_s^n$ be a zonoid.
Then equality holds in \eqref{eq:locallogbm} if and only if
\begin{enumerate}[1.]
\itemsep\abovedisplayskip
\item $K=C_1+\cdots+C_m$ for some $1\le m\le n$ and zonoids 
$C_1,\ldots,C_m$ such that $\dim(C_1)+\cdots+\dim(C_m)=n$; and
\item there exist $a_1,\ldots,a_m\ge 0$ such that
$L$ and $a_1C_1+\cdots+a_mC_m$ have the same supporting hyperplanes
in all $1$-extreme normal directions of $K$.
\end{enumerate}
\end{thm}

Theorem \ref{thm:main} does not suffice to conclude that the logarithmic 
Brunn-Minkowski inequality \eqref{eq:logbm} holds when $K,L$ are zonoids, 
as $K^{1-t}L^t$ is generally not a zonoid. Nonetheless, by combining 
Theorems \ref{thm:main}--\ref{thm:equality} with 
\cite[Theorem~2.1]{Mil21} we can deduce validity of the logarithmic 
Minkowski inequality \eqref{eq:logmink}, albeit without its equality 
cases. Some further implications will be given in section 
\ref{sec:impl}.

\begin{cor}
\label{cor:main}
Let $K\in\mathcal{K}_s^n$ be a zonoid. Then the logarithmic Minkowski 
inequality \eqref{eq:logmink} holds for all $L\in\mathcal{K}_s^n$.
\end{cor}

It appears somewhat unlikely that our results make major progress 
in themselves toward the full resolution of the logarithmic Brunn-Minkowski 
conjecture; as is the case for other well-known conjectures in convex 
geometry (see, e.g., \cite{GMR88}), zonoids form a very special class of 
convex bodies that provide only modest insight into the behavior of 
general convex bodies. Nonetheless, let us highlight several interesting 
features of the main results of this note:
\begin{enumerate}[$\bullet$] 
\itemsep\abovedisplayskip 
\item Theorem \ref{thm:main} possesses many nontrivial equality cases; 
therefore, in contrast to the setting of previous results in dimensions 
$n\ge 3$ for general $L\in\mathcal{K}_s^n$, the class of zonoids includes 
many extreme cases of the logarithmic Brunn-Minkowski conjecture.
(Theorem \ref{thm:equality} supports the conjectured equality cases
in \cite{BK20}.) 
\item Unlike in dimensions $n\ge 3$, every planar symmetric convex body is 
a zonoid. The $n=2$ case of the logarithmic Brunn-Minkowski conjecture 
that was settled in \cite{BLYZ12} may therefore be viewed in a new light 
as a special case of our results (modulo the nontrivial Theorem 
\ref{thm:logequiv}). In fact, the proof of Theorem \ref{thm:main} will 
work in a completely analogous manner for $n=2$ and $n\ge 3$. 
\item The $\ell_p^n$-ball is a zonoid for every $n$ and $2\le 
p\le\infty$ \cite[Theorem 6.6]{Bol69}.
Our results therefore capture as special cases all explicit 
examples of convex bodies $K$ for which \eqref{eq:logmink} and 
\eqref{eq:locallogbm} were previously known to hold.\footnote{%
	However, the methods of \cite{KM17,Mil21} provide complementary
	information that
	does not follow from our results.
	For example, the estimates of \cite{KM17} imply that for any
	$2<p<\infty$ and $n\ge n_0(p)$, all
	$K\in\mathcal{K}_s^n$ that are sufficiently close to the 
	$\ell_p^n$-ball in a quantitative sense satisfy the
	$L^q$-Minkowski inequality with $q=-\frac{1}{4}$.
	More generally, it is shown in \cite{Mil21} that for
        any $K\in\mathcal{K}_s^n$, there exists $K'\in\mathcal{K}_s^n$
        with $K\subseteq K'\subseteq 8K$ so that $K'$ satisfies the
        $L^q$-Minkowski inequality with $q=-\frac{1}{4}$.
}
\end{enumerate}
Before we proceed, let us briefly sketch some key ideas behind the proofs.

\subsection{}
\label{sec:sketch}

It was a fundamental insight of Hilbert \cite[Chapter XIX]{Hil12} that
mixed volumes of sufficiently smooth convex bodies admit a spectral 
interpretation. To this end, given any sufficiently smooth convex body $K
\in\mathcal{K}^n$, Hilbert constructs an elliptic differential 
operator $\mathscr{A}_K$ (see section \ref{sec:hilbert} for
a precise definition) and a measure $d\mu_K := 
\frac{1}{nh_K}dS_{K,\ldots,K}$ on $S^{n-1}$ with the 
following properties:
\begin{enumerate}[$\bullet$]
\itemsep\abovedisplayskip
\item
$\mathscr{A}_K$ defines a self-adjoint operator on $L^2(\mu_K)$ with
discrete spectrum.
\item
$\mathscr{A}_Kh_K=h_K$, that is,
$h_K$ is an eigenfunction with eigenvalue $1$.
\item 
$\V(L,M,K,\ldots,K) = \langle h_L,\mathscr{A}_K h_M\rangle$ for all
$L,M\in\mathcal{K}^n$.
\end{enumerate}
Using these properties, it is readily verified that \eqref{eq:minksecond} 
is equivalent to
\begin{equation}
\label{eq:hilbert}
	\langle f,\mathscr{A}_Kf\rangle\le 0
	\quad\mbox{for}\quad
	f = h_L - \frac{\langle 
	h_L,h_K\rangle}{\|h_K\|^2}
	h_K,
\end{equation}
where we denote by $\langle\cdot,\cdot\rangle$ and
$\|\cdot\|$ the inner product and norm of $L^2(\mu_K)$. As 
$f$ in \eqref{eq:hilbert} is the projection of $h_L$ on $\{h_K\}^\perp$, 
we obtain:

\begin{lem}[Hilbert]
\label{lem:hilbert}
Items 1--3 of Lemma \ref{lem:bmequiv} are equivalent to:
\begin{enumerate}[1.]
\setcounter{enumi}{3}
\item
For every sufficiently smooth convex body $K\in\mathcal{K}^n$,
any eigenfunction 
$\mathscr{A}_Kf=\lambda f$ with $\langle f, h_K\rangle=0$ has eigenvalue 
$\lambda\le 0$.
\end{enumerate}
\end{lem}

The condition of Lemma \ref{lem:hilbert} is optimal, as $\mathscr{A}_K$ 
always has eigenfunctions with eigenvalue $0$: any linear function 
$\ell(x)=h_{\{v\}}(x)=\langle v,x\rangle$ satisfies $\mathscr{A}_K\ell=0$. 
If we restrict attention to \emph{symmetric} convex bodies 
$K,L\in\mathcal{K}_s^n$, however, only \emph{even} functions $f(x)=f(-x)$ 
arise in \eqref{eq:hilbert}, and it is certainly possible that all even 
eigenfunctions of $\mathscr{A}_K$ have strictly negative eigenvalues. It 
was observed by Kolesnikov and Milman \cite{KM17} that the logarithmic 
Brunn-Minkowski conjecture may be viewed as a quantitative form of this 
phenomenon: as \eqref{eq:locallogbm} is equivalent to
$$
	\langle f,\mathscr{A}_Kf\rangle \le -\frac{1}{n-1}
	\|f\|^2
        \quad\mbox{for}\quad
        f = h_L - \frac{\langle
        h_L,h_K\rangle}{\|h_K\|^2}
        h_K,
$$
the following conclusion follows readily.

\begin{lem}[Kolesnikov-Milman]
\label{lem:km}
Items 1--3 of Theorem \ref{thm:logequiv} are equivalent to:
\begin{enumerate}[1.]
\setcounter{enumi}{3}
\item
For every sufficiently smooth symmetric convex body $K\in\mathcal{K}^n_s$,
any \emph{even} eigenfunction 
$\mathscr{A}_Kf=\lambda f$ with $\langle f, h_K\rangle=0$ has eigenvalue 
$\lambda\le -\frac{1}{n-1}$.
\end{enumerate}
\end{lem}

It is verified in \cite[Theorem 10.4]{KM17} that when $K$ is the 
$\ell_p^n$-ball for $2\le p<\infty$, any even eigenfunction of 
$\mathscr{A}_K$ orthogonal to $h_K$ has eigenvalue $\lambda$ with 
$n\lambda\to-\infty$ as $n\to\infty$. This shows that such $K$ satisfy the 
condition of Lemma \ref{lem:km} for large $n$, but does not explain the 
significance of the threshold $-\frac{1}{n-1}$.

A new approach to the study of the spectral properties of $\mathscr{A}_K$ 
was discovered by Shenfeld and the author in \cite{SvH18}. This approach, 
called the \emph{Bochner method} in view of its analogy to the classical 
Bochner method in differential geometry, has found several surprising 
applications both inside and outside convex geometry. The Bochner method 
was already used in \cite{SvH18} to provide new proofs of the 
Alexandrov-Fenchel inequality, a much deeper result of which 
\eqref{eq:minksecond} is a special case, and of the Alexandrov mixed 
discriminant inequality. Subsequent applications outside convexity include 
the proof of certain properties of Lorentzian polynomials in 
\cite{CKMS19,BL21} and the striking results of \cite{CP21}, where the 
method is used to prove numerous combinatorial inequalities. The paper 
\cite{Mil21} contains another application to the study of isomorphic 
variants of the $L^q$-Minkowski problem.

The proofs of Theorems \ref{thm:main}--\ref{thm:equality} provide yet 
another illustration of the utility of the Bochner method. By using a 
variation on the method of \cite{SvH18}, we obtain a ``Bochner identity'' 
which relates the spectral condition of Lemma \ref{lem:km} in dimension 
$n$ to the inequality \eqref{eq:locallogbm} in dimension $n-1$. The 
conclusion then follows by induction on the dimension. One interesting 
feature of this proof is that it provides an explanation for the appearance 
of the mysterious value $-\frac{1}{n-1}$ in Lemma \ref{lem:km}.\footnote{
	As was pointed out in \cite{KM17}, the eigenvalue $-\frac{1}{n-1}$
	is attained when $K$ is the cube, so that Lemma \ref{lem:km} may
	be interpreted as stating that the second eigenvalue of 
	$\mathscr{A}_K$ is maximized by the cube. This interpretation
	does not explain, however, \emph{why} this should be the case.
	In any case, there are many maximizers other than cubes, as
	is already illustrated by Theorem \ref{thm:equality}.
}
While the specific formulas derived in this note rely on the zonoid 
assumption, our approach may provide some hope that other variations on 
the Bochner method could lead to further progress toward the logarithmic 
Brunn-Minkowski conjecture.

\subsection{}

The rest of this note is organized as follows. In section 
\ref{sec:prelim}, we briefly recall some background from convex geometry 
that will be needed in the proofs, and we recall the basic idea behind the 
Bochner method as developed in \cite{SvH18}. Theorem \ref{thm:main} is 
proved in section \ref{sec:main}, and Theorem \ref{thm:equality} is 
proved in section \ref{sec:eq}. Finally, section \ref{sec:impl} spells out 
some implications of Theorem \ref{thm:main}, including the proof of 
Corollary \ref{cor:main}.

\section{Preliminaries}
\label{sec:prelim}

Throughout this note, we will use without comment the standard properties 
of mixed volumes and mixed area measures: that they are nonnegative, 
symmetric and multilinear in their arguments, and continuous under 
Hausdorff convergence. We refer to the monograph \cite{Sch14} for a 
detailed treatment, or to \cite[\S 2]{SvH18}, \cite[\S 4]{SvH22} for a 
brief review of such basic properties. The aim of this section is to 
recall some further notions that will play a central role in the sequel: 
the behavior of mixed volumes under projections, the construction of the 
Hilbert operator $\mathscr{A}_K$ for sufficiently smooth convex bodies, 
and the Bochner method of \cite{SvH18}.

The following notation will often be used: if $f=h_K-h_L$ is a difference
of support functions of convex bodies, then we define \cite[\S 5.2]{Sch14}
\begin{align*}
	&\V(f,C_1,\ldots,C_{n-1}) :=
	\V(K,C_1,\ldots,C_{n-1})-\V(L,C_1,\ldots,C_{n-1}),\\
	&S_{f,C_1,\ldots,C_{n-2}} := S_{K,C_1,\ldots,C_{n-2}}-
	S_{L,C_1,\ldots,C_{n-2}}.
\end{align*}
We similarly define $\V(f,g,C_1,\ldots,C_{n-2})$ by linearity when $f,g$ 
are differences of support functions, etc. Mixed volumes and area measures 
of differences of support functions are still symmetric and multilinear,
but need not be nonnegative.

\subsection{Projections and zonoids}

Let $E\subseteq\mathbb{R}^n$ be a linear subspace of dimension $k$, and 
let $C_1,\ldots,C_k$ be convex bodies in $E$. Then we denote by 
$\V(C_1,\ldots,C_k)$ and $S_{C_1,\ldots,C_{k-1}}$ the mixed volume and 
mixed area measure computed in $E\simeq\mathbb{R}^k$. We will often view 
$S_{C_1,\ldots,C_{k-1}}$ as a measure on $\mathbb{R}^n$ that is supported 
in $E$. The projection of a convex body $C$ in $\mathbb{R}^n$ onto $E$ 
will be denoted as $\proj_EC$.

The following basic formulas relate mixed volumes and mixed area measures 
of convex bodies to those of their projections.

\begin{lem}
\label{lem:proj}
For any $u\in S^{n-1}$ and $C_1,\ldots,C_{n-1}\in\mathcal{K}^n$, we have
\begin{align*}
	\frac{n}{2}\,\V([-u,u],C_1,\ldots,C_{n-1}) &=
	\V(\proj_{u^\perp}C_1,\ldots,\proj_{u^\perp}C_{n-1}),
	\\
	\frac{n-1}{2}\,S_{[-u,u],C_1,\ldots,C_{n-2}} &=
	S_{\proj_{u^\perp}C_1,\ldots,\proj_{u^\perp}C_{n-2}}.
\end{align*}
\end{lem}

\begin{proof}
The first identity is \cite[(5.77)]{Sch14}. To prove the second identity, 
note that the first identity may be rewritten using \eqref{eq:mixvolarea} 
as
$$
	\frac{1}{2}
	\int f \,
	dS_{[-u,u],C_1,\ldots,C_{n-2}} =
	\frac{1}{n-1}
	\int f \,
	dS_{\proj_{u^\perp}C_1,\ldots,\proj_{u^\perp}C_{n-2}}
$$
for $f=h_{C_{n-1}}$, where we used that $h_{\proj_EC}(u)=h_C(u)$ for $u\in E$. 
The identity extends by linearity to any difference of support functions
$f=h_K-h_L$. But as any $f\in C^2(S^{n-1})$ is of this form
(cf.\ Lemma \ref{lem:c2d2} below), the conclusion follows.
\end{proof}

A body $K\in\mathcal{K}^n_s$ is called a \emph{zonoid} if it is the 
limit of Minkowski sums of segments $[-u,u]$. The following equivalent
definition \cite[Theorem 3.5.3]{Sch14} is well known.

\begin{defn}
$K\in\mathcal{K}^n_s$ is called a \emph{zonoid} if
$$
	h_K(x) = \int h_{[-u,u]}(x)\,\eta(du)
$$
for some finite even measure $\eta$ on $S^{n-1}$, called the
\emph{generating measure} of $K$.
\end{defn}

The significance of zonoids for our purposes is that mixed volumes of 
zonoids can be expressed in terms of mixed volumes of projections by
Lemma \ref{lem:proj} and linearity. One simple 
illustration of this is the following fact.

\begin{lem}
\label{lem:planarzonoid}
In dimension $2$, any symmetric convex body $K\in\mathcal{K}^2_s$ is a 
zonoid with
$$
	h_K(x) = \frac{1}{4}\int 
	h_{[-u^\dagger,u^\dagger]}(x) \,S_K(du),
$$
where for $u\in S^1$ we denote by $u^\dagger\in S^1$ the 
clockwise rotation of $u$ by the angle $\frac{\pi}{2}$.
\end{lem}

\begin{proof}
Note first that $h_{[-u^\dagger,u^\dagger]}(x) =
|\langle u^\dagger,x\rangle| = |\langle u,x^\dagger\rangle| =
h_{[-x^\dagger,x^\dagger]}(u)$. Thus by Lemma \ref{lem:proj}, we can write 
for any $K\in \mathcal{K}^2$
$$
	\frac{1}{2}
	\int h_{[-u^\dagger,u^\dagger]}(x) \,S_K(du) =
	\V([-x^\dagger,x^\dagger],K) =
	\Vol(\proj_{\mathrm{span}\{x\}}K) =
	h_K(x)+h_K(-x).
$$
As $K\in\mathcal{K}^2_s$ is symmetric, we have
$h_K(x)+h_K(-x)=2h_K(x)$.
\end{proof}

\subsection{Smooth bodies and the Hilbert operator}
\label{sec:hilbert}

A support function $h_K$ may be viewed either as a function on 
$S^{n-1}$, or as a $1$-homogeneous function on $\mathbb{R}^n$. In 
particular, if $h_K$ is a $C^2$ function on $S^{n-1}$, then its gradient 
$\nabla h_K$ in $\mathbb{R}^n$ is $0$-homogeneous, and thus its Hessian 
$\nabla^2h_K(x)$ in $\mathbb{R}^n$ is a linear map from $x^\perp$ to 
itself. We denote the restriction of $\nabla^2h_K(x)$ to $x^\perp$ as 
$D^2h_K(x)$. For a general function $f\in C^2(S^{n-1})$, the restricted 
Hessian $D^2f$ is defined analogously by applying the above construction 
to the $1$-homogeneous extension of $f$.

We now recall the following basic facts \cite[\S 2.1]{SvH18}. Here we 
write $A\ge 0$ ($A>0$) to indicate that a symmetric matrix $A$ is 
positive semidefinite (positive definite).

\begin{lem}
\label{lem:c2d2}
Let $f\in C^2(S^{n-1})$. Then the following hold:
\begin{enumerate}[a.]
\itemsep\abovedisplayskip
\item $f=h_K$ for some convex body $K$ if and only if $D^2f\ge 0$.
\item For any convex body $L$ such that $h_L\in C^2(S^{n-1})$ and 
$D^2h_L>0$, there is a convex body $K$ and $a>0$ so that $f=a(h_K-h_L)$.
\end{enumerate}
\end{lem}

A particularly useful class of bodies is the following.

\begin{defn}
$K\in\mathcal{K}^n$ is of class $C^k_+$ ($k\ge 2$) if $h_K\in 
C^k(S^{n-1})$ and $D^2h_K>0$.
\end{defn}

For our purposes, the importance of $C^k_+$ bodies is that they admit 
certain explicit representations of mixed volumes and mixed area measures. 
To define these, let us first recall that the mixed discriminant 
$\D(A_1,\ldots,A_{n-1})$ of $(n-1)$-dimensional matrices 
$A_1,\ldots,A_{n-1}$ is defined by the formula
$$
	\det(\lambda_1A_1+\cdots+\lambda_mA_m)
	=\sum_{i_1,\ldots,i_{n-1}=1}^m
	\D(A_{i_1},\ldots,A_{i_{n-1}})\,\lambda_{i_1}\cdots
	\lambda_{i_{n-1}}
$$
in analogy with the definition of mixed volumes. Mixed 
discriminants are symmetric and multilinear in their arguments, and 
$\D(A_1,\ldots,A_{n-1})>0$
for $A_1,\ldots,A_{n-1}>0$.
Moreover, we have the Alexandrov mixed discriminant inequality
\begin{equation}
\label{eq:mixdisc}
	\D(A,B,M_1,\ldots,M_{n-3})^2 \ge
	\D(A,A,M_1,\ldots,M_{n-3})\,\D(B,B,M_1,\ldots,M_{n-3})
\end{equation}
whenever $B,M_1,\ldots,M_{n-3}\ge 0$ and $A$ is a symmetric matrix. For 
these and other facts about mixed discriminants, see
\cite[\S 2.3 and \S 4]{SvH18}.

With these definitions in place, we have the following
\cite[Lemma 4.7]{SvH22}.

\begin{lem}
\label{lem:mvc2p}
Let $C_1,\ldots,C_{n-1}\in\mathcal{K}^n$ be of class $C^2_+$. Then
\begin{align*}
	dS_{C_1,\ldots,C_{n-1}} &= \D(D^2h_{C_1},\ldots,
	D^2h_{C_{n-1}})\,d\omega,
	\\
	\V(K,C_1,\ldots,C_{n-1}) &=
	\frac{1}{n}\int h_K\,\D(D^2h_{C_1},\ldots,
        D^2h_{C_{n-1}})\,d\omega
\end{align*}
for any convex body $K$, where $\omega$ denotes the surface measure on 
$S^{n-1}$.
\end{lem}

We now introduce the spectral interpretation of mixed volumes due to 
Hilbert. For simplicity, we will only consider the special case that is 
needed in this note; the same construction applies to general mixed volumes 
(cf.\ \cite{SvH18}).

Fix a body $K\in\mathcal{K}^n$ of class $C^2_+$ with the origin in 
its interior (so that $h_K>0$). Then we define a measure $\mu_K$ on 
$S^{n-1}$ as
$$
	d\mu_K := \frac{1}{nh_K} dS_{K,\ldots,K} =
	\frac{1}{nh_K}
	\D(D^2h_K,\ldots,D^2h_K)\,d\omega,
$$
and define the second-order differential operator $\mathscr{A}_K$ on
$S^{n-1}$ as
$$
	\mathscr{A}_Kf := h_K\frac{\D(D^2f,D^2h_K,\ldots,D^2h_K)}{
	\D(D^2h_K,\ldots,D^2h_K)}
$$
for $f\in C^2(S^{n-1})$. The positivity of mixed discriminants of positive
definite matrices implies that $\mathscr{A}_K$ is elliptic. 
Standard facts of elliptic regularity theory therefore imply
the following, cf.\ \cite[\S 3]{SvH18} or \cite[Theorem 5.3]{KM17}:
\begin{enumerate}[$\bullet$]
\itemsep\abovedisplayskip
\item $\mathscr{A}_K$ extends to a self-adjoint operator 
on $L^2(\mu_K)$ with $\mathrm{Dom}(\mathscr{A}_K)=H^2(S^{n-1})$.
\item $\mathscr{A}_K$ has a discrete spectrum, that is, it
has a countable sequence of eigenvalues $\lambda_1>\lambda_2\ge
\lambda_3\ge
\cdots$ of tending to $-\infty$, and its eigenfunctions span $L^2(\mu_K)$.
\item $\lambda_1=1$ is a simple eigenvalue, 
whose eigenspace is spanned by $h_K$.
\end{enumerate}
These facts will be invoked in the sequel without further comment.

The point of this construction is that, by Lemmas \ref{lem:c2d2}
and \ref{lem:mvc2p}, we evidently have
$$
	\langle f,\mathscr{A}_Kg\rangle :=
	\int f\,\mathscr{A}_Kg\,d\mu_K
	=
	\V(f,g,K,\ldots,K)
$$
for any $f,g\in C^2(S^{n-1})$. Mixed volumes of this type may therefore be 
viewed as quadratic forms of the operator $\mathscr{A}_K$, which furnishes 
various geometric inequalities with a spectral interpretation as explained 
in section \ref{sec:sketch}.

When $K\in\mathcal{K}^n_s$ is symmetric, it is readily verified from the 
definitions that $\mu_K$ is an even measure, and that $\mathscr{A}_K$ 
leaves the spaces $L^2(\mu_K)_{\rm even}$ and $L^2(\mu_K)_{\rm odd}$ of 
even and odd functions on $S^{n-1}$ invariant. As 
$L^2(\mu_K)=L^2(\mu_K)_{\rm even}\oplus L^2(\mu_K)_{\rm odd}$, it follows 
that when $K\in\mathcal{K}^n_s$ any $f\in L^2(\mu_K)_{\rm even}$ can be 
expressed as a linear combination of the \emph{even} eigenfunctions of 
$\mathscr{A}_K$; cf.\ \cite[\S 5.1]{KM17}.

\subsection{The Bochner method}
\label{sec:bochnerold}

As was explained in section \ref{sec:sketch}, Minkowski's second inequality 
\eqref{eq:minksecond}, and thus the Brunn-Minkowski inequality, is 
equivalent to the statement that the second largest eigenvalue of 
$\mathscr{A}_K$ satisfies $\lambda_2\le 0$ (cf.\ Lemma~\ref{lem:hilbert}). 
This idea was exploited by Hilbert to give a spectral proof of the 
Brunn-Minkowski inequality by means of an eigenvalue continuity argument.

A new proof of the above spectral condition was discovered by 
Shenfeld and the author in \cite{SvH18}. This proof is based on the 
elementary fact that the condition $\lambda_2\le 0$ would follow directly 
from the Lichnerowicz condition 
\begin{equation} 
\label{eq:oldbochner}
	\langle \mathscr{A}_Kf,\mathscr{A}_Kf\rangle \ge
	\langle f,\mathscr{A}_Kf\rangle\quad\mbox{for all }f.
\end{equation}
Indeed, if $\mathscr{A}_Kf=\lambda f$, then \eqref{eq:oldbochner} yields 
$\lambda^2\ge\lambda$, i.e., $\lambda\ge 1$ or $\lambda\le 0$. As 
$1=\lambda_1>\lambda_2$ by elliptic regularity theory, 
the conclusion $\lambda_2\le 0$ follows.
While this is merely a reformulation of the problem, 
the beauty of \eqref{eq:oldbochner} is that it is an 
immediate consequence of the following identity that admits a one-line 
proof.

\begin{lem}[Bochner identity]
\label{lem:bochner}
Let $K\in\mathcal{K}^n$ be of class $C^2_+$ and let $f\in C^2$. Then
\begin{gather}
\label{eq:bochner}
\begin{aligned}
	&\langle \mathscr{A}_Kf,\mathscr{A}_Kf\rangle -
	\langle f,\mathscr{A}_Kf\rangle = \\
	&
	\int \frac{h_K}{n}\bigg\{
	\frac{\D(D^2f,D^2h_K,\ldots,D^2h_K)^2}{
	\D(D^2h_K,\ldots,D^2h_K)}-
	\D(D^2f,D^2f,D^2h_K,\ldots,D^2h_K)\bigg\} \,d\omega.
\end{aligned}
\end{gather}
\end{lem}

\begin{proof}
The identity is immediate from the definitions of $\mathscr{A}_K$
and $\mu_K$, and as
$$
	\langle f,\mathscr{A}_K f\rangle =
	\V(K,f,f,K,\ldots,K) 
	= \frac{1}{n}\int
	h_K\,\D(D^2f,D^2f,D^2h_K,\ldots,D^2h_K)\,d\omega
$$
by Lemma \ref{lem:mvc2p} and as mixed volumes 
are symmetric in their arguments.
\end{proof}

To deduce \eqref{eq:oldbochner}, it remains to recall that the integrand 
in \eqref{eq:bochner} is nonnegative by the following special case of the 
mixed discriminant inequality \eqref{eq:mixdisc}:
\begin{equation}
\label{eq:mixdsimple}
	\D(A,B,\ldots,B)^2  \ge
	\D(A,A,B,\ldots,B)\,
	\D(B,\ldots,B).
\end{equation}
In other words, the Bochner method reduces Minkowski's 
second inequality \eqref{eq:minksecond} to its 
linear-algebraic counterpart \eqref{eq:mixdsimple}.
This interpretation of the Bochner method will form the starting point 
for the main results of this note.

\begin{rem}
Lemma \ref{lem:bochner} is a trivial reformulation of the proof of 
\cite[Lemma 3.1]{SvH18}, as is explained in \cite[\S 6.3]{SvH18}. Moreover, 
it is observed there that in the special case that $K$ is the Euclidean 
ball, \eqref{eq:bochner} is precisely the classical (integrated) Bochner 
formula on $S^{n-1}$. One may therefore naturally view the above approach 
as an analogue of the Bochner method of differential geometry.

The identity \eqref{eq:bochner} was recently rediscovered by Milman 
\cite{Mil21}. A new insight of \cite{Mil21} is that \eqref{eq:bochner} may 
in fact be viewed as a true Bochner formula in the sense of differential 
geometry for any body $K$ of class $C^2_+$, by introducing a special 
centro-affine connection on $\partial K$. This interpretation does not 
appear to extend, however, to more general situations: for example, neither 
the more general identity that was used in \cite{SvH18} to prove the 
Alexandrov-Fenchel inequality, nor the ``Bochner identities'' of this 
note, are true Bochner formulas in the strictly formal sense, but should 
rather be viewed as a loose analogues of such a formula. The merits of 
taking a more liberal view on the Bochner method are illustrated by its 
diverse applications not only in convexity, but also in algebra and 
combinatorics \cite{SvH18,CKMS19,BL21,CP21}.
\end{rem}

\begin{rem}
The Bochner method should not be confused with a different 
method to prove Brunn-Minkowski inequalities that was developed by Reilly 
\cite{Rei80} and considerably refined by Kolesnikov and Milman in 
\cite{KM18,KM17}. The basis for Reilly's method is an integrated form of 
the classical Bochner formula on $\mathbb{R}^n$ (or on a manifold), combined with the solution of a certain Neumann 
problem. This method appears to be unrelated to the Bochner method for the 
operator $\mathscr{A}_K$.
\end{rem}

\section{Proof of Theorem \ref{thm:main}}
\label{sec:main}

The main step in the proof of Theorem \ref{thm:main} is the following 
analogue of \eqref{eq:oldbochner}.

\begin{thm}
\label{thm:superlich}
Let $K\in\mathcal{K}^n_s$ be a zonoid of class $C^2_+$ and
$f\in C^2(S^{n-1})_{\rm even}$. Then
\begin{equation}
\label{eq:logbochner}
        \langle \mathscr{A}_Kf,\mathscr{A}_Kf\rangle \ge
	\frac{n-2}{n-1}
        \langle f,\mathscr{A}_Kf\rangle +
	\frac{1}{n-1}\langle f,f\rangle.
\end{equation}
\end{thm}

Before we proceed, let us complete the proof of Theorem \ref{thm:main}.

\begin{proof}[Proof of Theorem \ref{thm:main}]
Let $K\in\mathcal{K}^n_s$ be a zonoid of class $C^2_+$.
As $\mathscr{A}_K$ is essentially self-adjoint on $C^2(S^{n-1})$, 
\eqref{eq:logbochner} extends directly to any
even function $f\in\mathrm{Dom}(\mathscr{A}_K)$. Thus if
$f$ is any even eigenfunction of $\mathscr{A}_K$ with eigenvalue 
$\lambda$, Theorem \ref{thm:superlich} yields 
$$
	\lambda^2\ge \frac{n-2}{n-1}\,\lambda+\frac{1}{n-1},
$$
i.e., $\lambda\ge 1$ or $\lambda\le-\frac{1}{n-1}$. But recall that
the largest eigenvalue of $\mathscr{A}_K$ is $\lambda_1=1$ and its
eigenspace is spanned by $h_K$. Thus any even eigenfunction $f$ of 
$\mathscr{A}_K$ that is orthogonal to $h_K$ must have eigenvalue
$\lambda\le -\frac{1}{n-1}$. In particular, as any $f\in L^2(\mu_K)_{\rm 
even}$ is in the linear span of the even eigenfunctions
of $\mathscr{A}_K$, we obtain
$$
	\langle f,\mathscr{A}_Kf\rangle \le
	-\frac{1}{n-1}\langle f,f\rangle
	\quad
	\mbox{whenever }f\in C^2(S^{n-1})_{\rm even},~
	\langle f,h_K\rangle = 0.
$$
For any $L\in\mathcal{K}^n_s$ of class $C^2_+$, we may now choose
$f = h_L - \frac{\langle h_L,h_K\rangle}{\langle h_K,h_K\rangle}\,h_K$
and use 
\begin{gather*}
	\langle h_L,\mathscr{A}_Kh_L\rangle =\V(L,L,K,\ldots,K),
	\qquad
	\langle h_L,h_L\rangle = \frac{1}{n}\int 
	\frac{h_L^2}{h_K}\,dS_{K,\ldots,K},\\
	\langle h_L,\mathscr{A}_Kh_K\rangle=\langle h_L,h_K\rangle = 
	\V(L,K,\ldots,K)
\end{gather*}
to conclude the validity of \eqref{eq:locallogbm} when
$K,L$ are of class $C^2_+$.

To conclude the proof, it suffices to show that for any
$K,L\in\mathcal{K}^n_s$ such that $K$ is a zonoid, there exist
$K_n,L_n\in\mathcal{K}^n_s$ of class $C^2_+$ such that $K_n$ is a zonoid 
and $K_n\to K$, $L_n\to L$ in the Hausdorff metric; the validity
of \eqref{eq:locallogbm} then follows by the continuity of mixed 
volumes and area measures. Both statements are classical; an
approximation of $L$ by $C^2_+$ bodies is given in \cite[\S 3.4]{Sch14},
while the approximation of $K$ may be performed, for example, by choosing
$h_{K_n} = \int h_{E_{n,u}}\,\eta(du)$ where $\eta$ is the generating 
measure of $K$ and $E_{n,u}$ are ellipsoids such that $E_{n,u}\to[-u,u]$.
\end{proof}

The remainder of this section is devoted to the proof of Theorem 
\ref{thm:superlich}. In essence, the inequality \eqref{eq:logbochner} will 
follow from a ``Bochner identity'' in the spirit of \eqref{eq:bochner}. 
However, rather than reducing the validity of \eqref{eq:locallogbm} to a 
linear algebraic analogue as was done in section \ref{sec:bochnerold}, the 
Bochner method will be used here to reduce \eqref{eq:locallogbm} in 
dimension $n$ to its validity in dimension $n-1$. The conclusion then 
follows by induction. As will be explained below, the structure of the 
induction also provides an explanation for the appearance of the 
mysterious value $-\frac{1}{n-1}$ in Lemma \ref{lem:km}.

\subsection{The induction step}

We begin with the following observation.

\begin{lem}
\label{lem:indlocal}
Let $n\ge 3$, $K\in\mathcal{K}^n_s$ be a zonoid, and
$f=h_M-h_{M'}$ for $M,M'\in\mathcal{K}^n_s$.
Assume that Theorem \ref{thm:main} has been 
proved in dimension $n-1$. Then
\begin{multline*}
	\frac{\V([-u,u],f,K,\ldots,K)^2}{\V([-u,u],K,\ldots,K)}
	\ge\\ 
	\frac{n-2}{n-1}\,\V([-u,u],f,f,K,\ldots,K) +
	\frac{1}{n(n-1)}\int \frac{f^2}{h_K}\,dS_{[-u,u],K,\ldots,K}
\end{multline*}
for every $u\in S^{n-1}$.
\end{lem}

\begin{proof}
Assume first that $K$ is a zonoid of class $C^2_+$ and that
$f\in C^2(S^{n-1})_{\rm even}$. Then
by Lemma \ref{lem:c2d2}, there exists a convex body $L\in\mathcal{K}_s^n$ 
of class $C^2_+$ and $a>0$ such that $f= a(h_L-h_K)$. By expanding the 
squares, the inequality in the statement is readily seen to be 
equivalent to the inequality
\begin{multline*}
	\frac{\V([-u,u],L,K,\ldots,K)^2}{\V([-u,u],K,\ldots,K)}
	\ge\\ \frac{n-2}{n-1}\,\V([-u,u],L,L,K,\ldots,K) +
	\frac{1}{n(n-1)}\int \frac{h_L^2}{h_K}\,dS_{[-u,u],K,\ldots,K}.
\end{multline*}
By Lemma \ref{lem:proj}, this is further equivalent to
\begin{multline*}
	\frac{\V(\proj_{u^\perp}L,\proj_{u^\perp}K,\ldots,
	\proj_{u^\perp}K)^2}{
	\V(\proj_{u^\perp}K,\ldots,
        \proj_{u^\perp}K)} \ge \\
	\frac{n-2}{n-1}\,
	\V(\proj_{u^\perp}L,\proj_{u^\perp}L,
	\proj_{u^\perp}K,\ldots,\proj_{u^\perp}K) +
	\frac{1}{(n-1)^2}
	\int \frac{h_{\proj_{u^\perp}L}^2}{h_{\proj_{u^\perp}K}}
	\,dS_{\proj_{u^\perp}K,\ldots,\proj_{u^\perp}K},
\end{multline*}
where we used that $h_{\proj_{u^\perp}L}(x)=h_L(x)$ for $x\in u^\perp$.
But as $\proj_{u^\perp}K$ is a zonoid, the latter inequality follows 
immediately from Theorem \ref{thm:main} in dimension $n-1$.
It remains to extend the conclusion to general $K$ and $f=h_M-h_{M'}$ by 
approximating $K,M,M'$ by $C^2_+$ bodies as in the proof of Theorem 
\ref{thm:main}.
\end{proof}

We are now ready to perform the induction step
in the proof of Theorem \ref{thm:superlich}.

\begin{prop}
\label{prop:induct}
Let $n\ge 3$, and assume that Theorem \ref{thm:main} has been proved in 
dimension $n-1$. Then the conclusion of Theorem \ref{thm:superlich}
holds in dimension $n$.
\end{prop}

\begin{proof}
Let $n\ge 3$, $f\in C^2(S^{n-1})_{\rm even}$, and
$K\in\mathcal{K}^n_s$ be a zonoid of class $C^2_+$ with generating measure 
$\eta$. We may write
$$
	\langle \mathscr{A}_K f,\mathscr{A}_K f\rangle =
	\frac{1}{n}
	\int 
	\int h_{[-u,u]}
	\frac{\D(D^2f,D^2h_K,\ldots,D^2h_K)^2}{
	\D(D^2h_K,\ldots,D^2h_K)}\,d\omega
	\,\eta(du)
$$
by the definitions of $\mathscr{A}_K,\mu_K$ and as
$h_K=\int h_{[-u,u]}\,\eta(du)$. Now note that
\begin{align*}
	&\frac{1}{n}\int h_{[-u,u]}
	\frac{\D(D^2f,D^2h_K,\ldots,D^2h_K)^2}{
	\D(D^2h_K,\ldots,D^2h_K)}\,d\omega \ge
	\\ &\qquad
	\frac{\big(\frac{1}{n}\int 
	h_{[-u,u]}\D(D^2f,D^2h_K,\ldots,D^2h_K)\,d\omega\big)^2}{
	\frac{1}{n}
	\int h_{[-u,u]}\D(D^2h_K,\ldots,D^2h_K)\,d\omega} 
	=
	\frac{\V([-u,u],f,K,\ldots,K)^2}{\V([-u,u],K,\ldots,K)}
\end{align*}
for any $u$
by Cauchy-Schwarz and Lemma \ref{lem:mvc2p}. We therefore obtain
\begin{align*}
	&\langle \mathscr{A}_K f,\mathscr{A}_K f\rangle \ge
	\int \frac{\V([-u,u],f,K,\ldots,K)^2}{\V([-u,u],K,\ldots,K)}
	\,\eta(du)
	\\
	&\quad\ge
	\int
	\bigg(
	\frac{n-2}{n-1}\,\V([-u,u],f,f,K,\ldots,K) +
	\frac{1}{n(n-1)}\int \frac{f^2}{h_K}\,dS_{[-u,u],K,\ldots,K}
	\bigg)\eta(du)
	\\ &\quad=
	\frac{n-2}{n-1}\langle f,\mathscr{A}_Kf\rangle +
	\frac{1}{n-1}\langle f,f\rangle
\end{align*}
using Lemma \ref{lem:indlocal} and $h_K=\int h_{[-u,u]}\,\eta(du)$.
\end{proof}

\begin{rem}
\label{rem:protoequality}
While we find it cleaner to formulate the proof of Proposition 
\ref{prop:induct} in terms of inequalities, one may in principle interpret 
this proof as arising from a Bochner \emph{identity} in 
the spirit of \eqref{eq:bochner}: indeed, combining the proofs of Lemma 
\ref{lem:indlocal} and Proposition \ref{prop:induct} yields
for $f=a(h_L-h_K)$
\begin{align*}
	&\langle \mathscr{A}_K f,\mathscr{A}_K f\rangle -
	\frac{n-2}{n-1}\langle f,\mathscr{A}_Kf\rangle -
	\frac{1}{n-1}\langle f,f\rangle =
	\\
	&
	\quad\int 
	\int \frac{h_{[-u,u]}}{h_K}
	\bigg(
	\mathscr{A}_Kf
	-
	\frac{\V([-u,u],f,K,\ldots,K)}{\V([-u,u],K,\ldots,K)}h_K
	\bigg)^2\,
	d\mu_K\,\eta(du) + \mbox{} \\
	&\quad\frac{2a^2}{n}\int\bigg(
	\frac{\V(\proj_{u^\perp}L,\proj_{u^\perp}K,\ldots,
	\proj_{u^\perp}K)^2}{
	\V(\proj_{u^\perp}K,\ldots,
        \proj_{u^\perp}K)} -
	\frac{n-2}{n-1}\,
	\V(\proj_{u^\perp}L,\proj_{u^\perp}L,
	\proj_{u^\perp}K,\ldots,\proj_{u^\perp}K) 
	\\ &\qquad\qquad\qquad -
	\frac{1}{(n-1)^2}
	\int \frac{h_{\proj_{u^\perp}L}^2}{h_{\proj_{u^\perp}K}}
	\,dS_{\proj_{u^\perp}K,\ldots,\proj_{u^\perp}K}	
	\bigg)\,
	\eta(du),
\end{align*}
where the two terms on the right-hand side are the deficits of
the two inequalities used in the proof (the Cauchy-Schwarz 
inequality and \eqref{eq:locallogbm} in dimension $n-1$, respectively).
While it would be difficult to recognize this identity as a Bochner 
formula in the sense of differential geometry, it 
plays precisely the same role in the present proof as the Bochner
identity \eqref{eq:bochner} in section \ref{sec:bochnerold}.

Let us further note that an even eigenfunction $\mathscr{A}_Kf=\lambda 
f$ yields equality in \eqref{eq:logbochner} if and only if $\lambda=1$ or 
$\lambda=-\frac{1}{n-1}$. When this is the case, the right-hand side of 
the above Bochner identity must vanish. It then follows from the first 
term on the right that $f$ must be proportional to $h_K$, so that 
$\lambda=1$. In other words, when the zonoid $K$ is of class $C^2_+$, any 
even eigenfunction that is orthogonal to $h_K$ has eigenvalue strictly 
less than $-\frac{1}{n-1}$, and thus no nontrivial equality cases can 
arise in \eqref{eq:locallogbm}. 
However, nontrivial equality cases can arise when $K$ is nonsmooth, which 
will be analyzed in section \ref{sec:eq} by a variation on the above 
argument. 
\end{rem}

\begin{rem}
At first sight, the formulation of the spectral condition of Lemma 
\ref{lem:km} is rather mysterious: what is the significance of the
special value $-\frac{1}{n-1}$? The present proof provides one explanation 
for the appearance of this value: the constants in
\eqref{eq:logbochner} in dimension $n$ are precisely the same as those 
that appear in \eqref{eq:locallogbm} in dimension $n-1$, so that
the preservation of the
sharp threshold $\lambda\le -\frac{1}{n-1}$ by induction on the dimension
$n$ is explained by the quadratic relation \eqref{eq:logbochner}.
\end{rem}

\subsection{The induction base}

By Proposition \ref{prop:induct} and induction on the dimension, the proof 
of Theorem \ref{thm:superlich} will be complete in any dimension $n\ge 3$ 
once we establish its validity in dimension $n=2$. The latter is already 
known, however, by the results of \cite{BLYZ12} and Theorem 
\ref{thm:logequiv}. On the other hand, as we 
will presently explain, the $n=2$ case may also be established directly by 
exactly the same method as was used in the proof of Proposition 
\ref{prop:induct}. This shows, in particular, that the Bochner method 
provides a unified explanation for the validity of Theorem \ref{thm:main} 
in every dimension.

\begin{lem}
\label{lem:planar}
The conclusion of Theorem \ref{thm:superlich} holds in dimension $n=2$.
\end{lem}

\begin{proof}
Let $f\in C^2(S^1)_{\rm even}$, and let $K\in\mathcal{K}^2_s$ be a zonoid
of class $C^2_+$.
Applying the Cauchy-Schwarz inequality as in the proof of
Proposition \ref{prop:induct} yields
$$
	\langle \mathscr{A}_Kf,\mathscr{A}_Kf\rangle
	\ge
	\frac{1}{4}
	\int \frac{\V([-u^\dagger,u^\dagger],f)^2}{
	\V([-u^\dagger,u^\dagger],K)}\,S_K(du),
$$
where we used Lemma \ref{lem:planarzonoid} to compute the generating 
measure of a planar zonoid. However, as was observed in the proof of
Lemma \ref{lem:planarzonoid}, we have
$$
	\V([-u^\dagger,u^\dagger],K) =
	2h_K(u),\qquad
	\V([-u^\dagger,u^\dagger],f) = 2f(u)
$$
(the latter follows as $f=a(h_L-h_K)$ for some $L\in\mathcal{K}^2_s$ by 
Lemma \ref{lem:c2d2}). Thus
$$
	\langle \mathscr{A}_Kf,\mathscr{A}_Kf\rangle \ge
	\frac{1}{2} \int
	\frac{f^2}{h_K}\,dS_K =
	\langle f,f\rangle,
$$
concluding the proof.
\end{proof}

\section{Proof of Theorem \ref{thm:equality}}
\label{sec:eq}

As was already noted in Remark \ref{rem:protoequality}, we may expect in 
principle that one may deduce the equality cases of \eqref{eq:locallogbm} 
by a careful analysis of the Bochner method. The immediate problem with 
this approach is that the most basic object that appears in the Bochner 
method---the Hilbert operator $\mathscr{A}_K$---is not even well defined
unless $K$ is of class $C^2_+$, and no nontrivial equality cases can arise 
in that setting. We will nonetheless pursue this strategy in the present 
section to settle the equality cases. This is possible, in essence, 
because it suffices for the purposes of characterizing equality to replace
$\mathscr{A}_Kf$ by $-\frac{1}{n-1}f$ in the Bochner identity, in which 
case the relevant formulas make sense also in nonsmooth situations.

We begin by making the latter idea precise in section \ref{sec:eqcond}. We 
subsequently show in section \ref{sec:bochnereq} what information on the 
equality cases may be extracted from the Bochner method. The proof of 
Theorem \ref{thm:equality} will be completed in section \ref{sec:pfeq}.

\subsection{The equality condition}
\label{sec:eqcond}

Before we proceed to the analysis of the equality cases, we state
a slight generalization of \eqref{eq:locallogbm} that will be needed in the 
sequel.

\begin{lem}
\label{lem:hedgehog}
Let $K\in\mathcal{K}^n_s$ be a zonoid. Then
$$
	\frac{\V(f,K,\ldots,K)^2}{\Vol(K)} \ge 
	\frac{n-1}{n}\,
	\V(f,f,K,\ldots,K)
	+
	\frac{1}{n^2} \int \frac{f^2}{h_K}\,dS_{K,\ldots,K}
$$
holds whenever $f=h_L-h_M$ for some $L,M\in\mathcal{K}^n_s$.
\end{lem}

\begin{proof}
This follows from Theorem~\ref{thm:main} as in the proof of Lemma 
\ref{lem:indlocal}.
\end{proof}

We can now obtain a basic reformulation of the equality condition in 
\eqref{eq:locallogbm}. The method is due to Alexandrov \cite[pp.\ 
80--81]{Ale96}.

\begin{lem}
\label{lem:alex}
For any $L\in\mathcal{K}_s^n$ and any zonoid $K\in\mathcal{K}^n_s$, the
following are equivalent:
\begin{enumerate}[1.]
\item Equality holds in \eqref{eq:locallogbm}, that is,
$$
	\frac{\V(L,K,\ldots,K)^2}{\Vol(K)} =
	\frac{n-1}{n}\,
	\V(L,L,K,\ldots,K)
	+
	\frac{1}{n^2} \int \frac{h_L^2}{h_K}\,dS_{K,\ldots,K}.
$$
\item There exists $a>0$ so that
$f=h_L-ah_K$ satisfies 
$$
	h_K\,dS_{f,K,\ldots,K} = 
	-\frac{1}{n-1}\,f\,dS_{K,\ldots,K}.
$$
\end{enumerate}
\end{lem}

\begin{proof}
We first prove that $2\Rightarrow 1$.
Integrating condition $2$ yields
$\V(f,K,\ldots,K)=0$, while multiplying condition $2$ by
$\frac{f}{h_K}$ and integrating yields
$$
	\frac{n-1}{n}\,
	\V(f,f,K,\ldots,K) = -\frac{1}{n^2}
	\int \frac{f^2}{h_K}\,dS_{K,\ldots,K}.
$$
We therefore obtain
$$
	\frac{\V(f,K,\ldots,K)^2}{\Vol(K)} =
	\frac{n-1}{n}\,
	\V(f,f,K,\ldots,K)
	+
	\frac{1}{n^2} \int \frac{f^2}{h_K}\,dS_{K,\ldots,K},
$$
and condition $1$ follows using $f=h_L-ah_K$ and expanding the squares.

We now prove the converse implication $1\Rightarrow 2$. Let $g\in 
C^2(S^{n-1})_{\rm even}$ and define
$$
	\beta(t) :=
	\frac{\V(g_t,K,\ldots,K)^2}{\Vol(K)} -
	\frac{n-1}{n}\,
	\V(g_t,g_t,K,\ldots,K)
	-
	\frac{1}{n^2} \int \frac{g_t^2}{h_K}\,dS_{K,\ldots,K}
$$
where $g_t:=h_L+tg$.
Condition $1$ implies $\beta(0)=0$, while Lemma \ref{lem:hedgehog} implies
$\beta(t)\ge 0$ for all $t$. Thus $\beta$ is minimized at 
zero, so that $\beta'(0)=0$ yields
$$
	\int g\,dS_{f,K,\ldots,K} = 
	-\frac{1}{n-1} \int g \,\frac{f}{h_K}\,dS_{K,\ldots,K}
$$
with $f = h_L - \frac{\V(L,K,\ldots,K)}{\Vol(K)}h_K$.
As $g\in
C^2(S^{n-1})_{\rm even}$ is arbitrary and as
$dS_{f,K,\ldots,K}$ and $\frac{f}{h_K}\,dS_{K,\ldots,K}$ are even measures, 
condition $2$ follows.
\end{proof}

It follows from the definition of the Hilbert operator $\mathscr{A}_K$ that 
when $K,L$ are of class $C^2_+$, Lemma \ref{lem:alex} states precisely that 
equality holds in \eqref{eq:locallogbm} if and only if 
$\mathscr{A}_Kf=-\frac{1}{n-1}f$ for $f=h_L-ah_K$. The point of Lemma 
\ref{lem:alex} is that the same characterization can be formulated for 
nonsmooth bodies in the sense of measures. The latter will suffice to apply 
the Bochner method to study the equality cases.

\subsection{The Bochner method revisited}
\label{sec:bochnereq}

Using Lemma \ref{lem:alex}, we can now essentially repeat the proof of 
Proposition \ref{prop:induct} in the present setting to extract a necessary 
condition for equality in \eqref{eq:locallogbm} from the Bochner method. 

\begin{lem}
\label{lem:eqbochner}
Let $K\in\mathcal{K}^n_s$ be a zonoid with generating measure $\eta$, and 
let $L\in\mathcal{K}^n_s$ be such 
that equality holds in \eqref{eq:locallogbm}. Then for every
$u\in\supp\eta$, there exists $c(u)\ge 0$ such that
$h_L(x)=c(u)h_K(x)$ for all $x\in\supp S_{K,\ldots,K}$ with
$|\langle u,x\rangle|>0$.
\end{lem}

\begin{proof}
Let $a>0$ be such that $f=h_L-ah_K$ satisfies the second condition of 
Lemma \ref{lem:alex}, that is, $f\,dS_{K,\ldots,K} = 
-(n-1)\,h_K\,dS_{f,K,\ldots,K}$. Then we have
\begin{align*}
	\int \frac{f^2}{h_K}\,dS_{K,\ldots,K} &=
	\int
	\int
	\frac{f^2}{h_K^2}\, 
	h_{[-u,u]}\,dS_{K,\ldots,K}\,\eta(du) 
	\\ 
	&\ge
	\int
	\frac{
	\big(\int \frac{f}{h_K}\,h_{[-u,u]}\,dS_{K,\ldots,K}\big)^2}{
	\int h_{[-u,u]}\,dS_{K,\ldots,K}}
	\,\eta(du)
	\\
	&=
	n(n-1)^2
	\int
	\frac{\V([-u,u],f,K,\ldots,K)^2}{
	\V([-u,u],K,\ldots,K)}
	\,\eta(du) \\
	&\ge
	n(n-1)(n-2) \V(f,f,K,\ldots,K) 
	+
	(n-1) \int \frac{f^2}{h_K}\,dS_{K,\ldots,K}
	\\
	&=
	\int \frac{f^2}{h_K}\,dS_{K,\ldots,K}.
\end{align*}
Here we used $h_K=\int h_{[-u,u]}\,\eta(du)$ in the first line; the 
Cauchy-Schwarz inequality in the second line; the condition of
Lemma \ref{lem:alex} in the third line; Lemma \ref{lem:indlocal} in the
fourth line (or by the proof of Lemma \ref{lem:planar} for $n=2$);
and the fifth line follows as
$$
	\V(f,f,K,\ldots,K) = 
	\frac{1}{n}\int f\,dS_{f,K,\ldots,K} =
	-\frac{1}{n(n-1)}\int \frac{f^2}{h_K}\,dS_{K,\ldots,K}
$$
by the condition of Lemma \ref{lem:alex}.

Consequently, both inequalities used above must hold with equality.
In particular, we have equality in the Cauchy-Schwarz inequality
$$
	\int
	\frac{f^2}{h_K^2}\, h_{[-u,u]}\,dS_{K,\ldots,K}
	=
	\frac{
	\big(\int \frac{f}{h_K}\,h_{[-u,u]}\,dS_{K,\ldots,K}\big)^2}{
	\int h_{[-u,u]}\,dS_{K,\ldots,K}}
$$
for every $u\in\supp\eta$. By the equality condition of the Cauchy-Schwarz 
inequality, this implies that for every $u\in\supp\eta$, there is a 
constant $c'(u)$ so that $f(x)=c'(u)h_K(x)$ for every $x\in\supp 
S_{K,\ldots,K}$ with $h_{[-u,u]}(x)=|\langle u,x\rangle|>0$.
But as $f=h_L-ah_K$, the conclusion follows with $c(u)=a+c'(u)$. (Note 
that it must be the case that $c(u)\ge 0$ as $h_K,h_L$ are positive 
functions.)
\end{proof}

\begin{rem}
The proof of Lemma \ref{lem:eqbochner} actually provides more information
than is expressed in its statement: not only do we get equality
in Cauchy-Schwarz, but we also get equality in the application
of Lemma \ref{lem:indlocal}. In particular, this implies that if equality
holds in \eqref{eq:locallogbm} for given $K,L$ in dimension $n$,
then the projections $\proj_{u^\perp}K,\proj_{u^\perp}L$ must also yield
equality in \eqref{eq:locallogbm} in dimension $n-1$ for every
$u\in\supp\eta$.
It is a curious feature of the present problem that the latter information
will not be needed to characterize the equality cases: the equality
condition in Cauchy-Schwarz will already suffice to fully characterize the 
equality cases of \eqref{eq:locallogbm}.
\end{rem}

\subsection{Characterization of equality}
\label{sec:pfeq}

We are now ready to proceed to the proof of Theorem \ref{thm:equality}. 
The main difficulty is to show that the stated conditions are necessary 
for equality, which will be deduced from Lemma \ref{lem:eqbochner}.

In the proof of the following result, we will encounter graphs that may 
have an uncountable number of vertices and edges. The standard properties 
of graphs that will be used in the proof---chiefly that a graph can be 
partitioned into its connected components---are valid at this level of 
generality; cf.\ \cite[Chapter 2]{Ore62}.

\begin{prop}
\label{prop:eqnecess}
Let $K\in\mathcal{K}^n_s$ be a zonoid, and let $L\in\mathcal{K}^n_s$ be 
such that equality holds in \eqref{eq:locallogbm}. Then there exist $1\le 
m\le n$, $a_1,\ldots,a_m\ge 0$, and zonoids $C_1,\ldots,C_m$ with 
$\dim(C_1)+\cdots+\dim(C_m)=n$ so that $K=C_1+\cdots+C_m$ and
$$
	h_L(x)=h_{a_1C_1+\cdots+a_mC_m}(x)
	\mbox{ for all }
	x\in\supp S_{K,\ldots,K}.
$$
\end{prop}

\begin{proof}
We define a graph $(V,E)$ as follows:
\begin{enumerate}[$\bullet$]
\itemsep\abovedisplayskip
\item The vertices are $V=\supp\eta$, where $\eta$ denotes the generating 
measure of $K$.
\item There is an edge $\{u,v\}\in E$ between $u,v\in V$ if and only
if there exists $x\in\supp S_{K,\ldots,K}$ such that
$|\langle u,x\rangle|>0$ and $|\langle v,x\rangle|>0$.
\end{enumerate}
Denote by $V=\bigsqcup_{i\in I}V_i$ the partition of $V$ into
its connected components $V_i$.

For any edge $\{u,v\}\in E$, Lemma \ref{lem:eqbochner} implies that
$$
	c(u)h_K(x)=h_L(x) = c(v)h_K(x)
$$
for some $x\in\supp S_{K,\ldots,K}$.
As $h_K(x)>0$, it follows that $c(u)=c(v)$. In particular, the value of 
$c(u)$ must be constant on each connected component. In the sequel, 
we will denote this value as $c(u)=a_i$ for $u\in V_i$.

Next, we make a key observation.

\begin{claim}
For every $x\in\supp S_{K,\ldots,K}$, there exists $i\in I$ so that
$x\perp V_j$ for all $j\ne i$.
\end{claim}

\begin{proof}
We can assume that $x\in\supp S_{K,\ldots,K}$ satisfies $|\langle 
u,x\rangle|>0$
for some $i\in I$, $u\in V_i$, as otherwise the conclusion is trivial. 
But then we must have
$|\langle v,x\rangle|=0$ for all $j\ne i$, $v\in V_j$, as distinct 
connected components have no edge between them.
\end{proof}

We also need the following.

\begin{claim}
For every $u\in S^{n-1}$, there exists $x\in\supp S_{K,\ldots,K}$ so that
$|\langle u,x\rangle|>0$.
\end{claim}

\begin{proof}
If the conclusion were false, there would exist some $u\in S^{n-1}$ such 
that $0=\int |\langle u,x\rangle|\,S_{K,\ldots,K}(dx) = 
2\,\Vol(\proj_{u^\perp}K)$ by Lemma \ref{lem:proj}. The latter is
impossible as $K\in\mathcal{K}^n_s$ is assumed to have nonempty interior.
\end{proof}

The above two claims imply that distinct $V_i$ must lie in
linearly independent subspaces $\mathrm{L}_i=\mathop{\mathrm{span}} V_i$.
Indeed, if this is not so, then there exists $z\in S^{n-1}$ so that
$$
	z = t_1u_1+\cdots+t_ku_k = s_1v_1+\cdots s_lv_l 
$$
for some $k,l\ge 1$, $i\in I$, $u_1,\ldots,u_k\in V_i$,
$v_1,\ldots,v_l\in \bigcup_{j\ne i}V_j$, 
$t_1,\ldots,t_k,s_1,\ldots,s_l\ne 0$.
By the second claim there exists $x\in \supp S_{K,\ldots,K}$ so that
$|\langle z,x\rangle|>0$. But by the first claim we must then have
$x\perp v_1,\ldots,u_l$, which entails a contradiction. It follows,
in particular, that there can be at most $n$ connected components,
so we can write $I=\{1,\ldots,m\}$ for some $1\le m\le n$.

We now define zonoids $C_1,\ldots,C_m$ as
$$
	h_{C_i} = \int_{\mathrm{L}_i} h_{[-u,u]}\,\eta(du).
$$
As $\mathrm{L}_1,\ldots,\mathrm{L}_m$ are linearly independent and 
$\supp\eta = V \subseteq S^{n-1}\cap(\mathrm{L}_1\cup\cdots\cup 
\mathrm{L}_m)$
$$
	h_{C_1}+\cdots+h_{C_m} =
	\int h_{[-u,u]}\,\eta(du) = h_K,
$$
that is, $K=C_1+\cdots+C_m$. Moreover, as 
$\mathrm{L}_1,\ldots,\mathrm{L}_m$ are linearly 
independent and $K$ has nonempty interior, we must have 
$\dim(C_1)+\cdots+\dim(C_m)=n$. 

Finally, let $x\in \supp S_{K,\ldots,K}$. By the first claim above, there
exists $1\le i\le m$ so that $h_{C_j}(x)=0$ for all $j\ne i$. As this 
implies that $h_{C_i}(x)=h_K(x)>0$, there must exist $u\in V_i$ so that
$|\langle u,x\rangle|>0$. Recalling that $c(u)=a_i$ for $u\in V_i$,
we obtain
$$
	h_L(x) = a_i h_K(x) = 
	a_i h_{C_i}(x) = 
	a_1 h_{C_1}(x)+\cdots+a_m h_{C_m}(x)
$$
by Lemma \ref{lem:eqbochner}. As this holds for any 
$x\in\supp S_{K,\ldots,K}$, the proof is 
complete.
\end{proof}

Before we complete the proof, we must verify the basic case of equality.

\begin{lem}
\label{lem:eqbasic}
Suppose that $K=C_1+\cdots+C_m$ for some convex bodies
$C_1,\ldots,C_m$ such that $\dim(C_1)+\cdots+\dim(C_m)=n$, and that
$L=a_1C_1+\cdots+a_mC_m$ for some $a_1,\ldots,a_m\ge 0$. Then
equality holds in \eqref{eq:locallogbm}.
\end{lem}

\begin{proof}
By \cite[Theorem 5.1.8]{Sch14}, the condition 
$\dim(C_1)+\cdots+\dim(C_m)=n$ implies that we have
$\V(C_{i_1},\ldots,C_{i_n})>0$ if and only if each index $1\le j\le m$
appears exactly $\dim(C_j)$ times among $(i_1,\ldots,i_n)$. Thus
for any $b_1,\ldots,b_m\ge 0$
$$
	\Vol(b_1C_1+\cdots+b_mC_m)=
	\sum_{i_1,\ldots,i_n=1}^m b_{i_1}\cdots b_{i_n}
	\V(C_{i_1},\ldots,C_{i_n})
	=
	\Gamma\, b_1^{\dim(C_1)}\cdots b_m^{\dim(C_m)}
$$
for some constant $\Gamma$ that depends only on $C_1,\ldots,C_m$.
Therefore
\begin{align*}
	0 &=
	-\frac{\Vol(K)}{n^2}
	\frac{d^2}{dt^2} \log \Vol(e^{t a_1} C_1 + \cdots + e^{t a_m} C_m)
	\bigg|_{t=0} \\
	&=
	\frac{\V(L,K,\ldots,K)^2}{\Vol(K)}
	-\frac{n-1}{n}\,\V(L,L,K,\ldots,K)
	-\frac{1}{n^2}\int h_{a_1^2  C_1 + \cdots + a_m^2 C_m}
	dS_{K,\ldots,K}.
\end{align*}
Now note that if $\int h_{C_i}\, dS_{C_{i_1},\ldots,C_{i_{n-1}}}>0$, then
using \cite[Theorem 5.1.8]{Sch14} as above shows that
$\int h_{C_j}\, dS_{C_{i_1},\ldots,C_{i_{n-1}}}=0$ for all $j\ne i$. 
In particular, as we have
$S_{K,\ldots,K}=\sum_{i_1,\ldots,i_{n-1}}S_{C_{i_1},\ldots,C_{i_{n-1}}}$,
this implies that for every $x\in\supp S_{K,\ldots,K}$, there exists an 
index $i$ so that $h_{C_j}(x)=0$ for all $j\ne i$. It follows readily that
$$
	h_{a_1^2C_1+\cdots+a_m^2C_m}(x)= \frac{h_L(x)^2}{h_K(x)}
	\quad\mbox{for all }x\in\supp S_{K,\ldots,K},
$$
and the proof is complete.
\end{proof}

We can now complete the proof of the necessity part of Theorem 
\ref{thm:equality}. In the proof, we use some nontrivial facts that do not 
appear elsewhere in this note.

\begin{proof}[Proof of Theorem \ref{thm:equality}]
We first prove sufficiency. Suppose that $K=C_1+\ldots+C_m$ for
bodies $C_1,\ldots,C_m$ with $\dim(C_1)+\cdots+\dim(C_m)=n$, and
that $L$ and $L':=a_1C_1+\cdots+a_mC_m$ have the same supporting 
hyperplanes in all $1$-extreme normal directions of $K$. The latter
implies by \cite[Theorem 4.5.3 and Lemma 7.6.15]{Sch14} that
\begin{equation}
\label{eq:support}
	h_L(x)=h_{L'}(x)\quad\mbox{for all }x\in\supp S_{M,K,\ldots,K}
\end{equation}
for any convex body $M$. In particular, every term in 
\eqref{eq:locallogbm} is unchanged if we replace $L$ by $L'$.
Thus equality holds in \eqref{eq:locallogbm} by Lemma \ref{lem:eqbasic}.

We now prove necessity. Suppose equality holds in 
\eqref{eq:locallogbm}. Then Proposition \ref{prop:eqnecess} provides 
$C_1,\ldots,C_m$ that satisfy all the required properties by construction 
except the last one: that is, what remains to be shown is that $L$ and 
$L':=a_1C_1+\cdots+a_mC_m$ have the same supporting hyperplanes in all 
$1$-extreme normal directions of $K$.

Let us write $f:=h_{L}-h_{L'}$. By Proposition \ref{prop:eqnecess}, we 
have $f=0$ on $\supp S_{K,\ldots,K}$. Moreover, as we clearly have
$L'+C = (\max_k a_k)K$ for a convex body $C$, it follows that
$\supp S_{L',K,\ldots,K}\subseteq \supp S_{K,\ldots,K}$ and thus
$f=0$ on $\supp S_{L',K,\ldots,K}$ as well. Substituting $h_L=h_{L'}+f$ 
into \eqref{eq:locallogbm} and using that both $L$ and $L'$ yield equality in
\eqref{eq:locallogbm} (by assumption and by Lemma \ref{lem:eqbasic}, 
respectively), we can readily compute
$$
	\V(f,K,\ldots,K)=0,\qquad
	\V(f,f,K,\ldots,K) = 0.
$$
Using that $f=h_{L}-h_{L'}$, this implies that we have equality
$$
	\V(L,L',K,\ldots,K)^2 = \V(L,L,K,\ldots,K)\,
	\V(L',L',K,\ldots,K)
$$
in Minkowski's quadratic inequality. By the main result of 
\cite{SvH22}, it follows that $L$ and $aL'+v$ have the same supporting
hyperplanes in all $1$-extreme normal directions of $K$ for some
$a\ge 0$, $v\in\mathbb{R}^n$. But as $L,L'$ are symmetric we must
have $v=0$, while $\V(f,K,\ldots,K)=0$ and \eqref{eq:support} imply $a=1$. 
This concludes the proof.
\end{proof}

\section{Implications}
\label{sec:impl}

As we recalled in Theorem \ref{thm:logequiv}, the validity of the local 
logarithmic Brunn-Minkowski inequality \eqref{eq:locallogbm} for 
\emph{all} $K\in\mathcal{K}^n_s$ is equivalent to the validity of the 
logarithmic Brunn-Minkowski and the logarithmic Minkowski inequalities. 
The proof of these facts is based on several recent deep results on 
uniqueness in the $L^q$-Minkowski problem for $q<1$. While this 
equivalence does not hold for fixed $K\in\mathcal{K}^n_s$, it is explained 
in \cite[\S 2.4]{Mil21} that the theory behind Theorem \ref{thm:logequiv} 
still yields nontrivial implications when \eqref{eq:locallogbm} is known 
to hold in a sufficiently rich sub-class of $\mathcal{K}^n_s$. The aim of 
the final section of this note is to investigate what conclusions may be 
drawn by combining these results with Theorems 
\ref{thm:main}--\ref{thm:equality}.

We begin with the proof of Corollary \ref{cor:main}.

\begin{proof}[Proof of Corollary \ref{cor:main}]
By a routine approximation argument as in the proof of Theorem 
\ref{thm:main}, it suffices to prove the validity of \eqref{eq:logmink}
for $K\in\mathcal{K}_s^n$ that are zonoids of class $C^\infty_+$. Let us 
fix such a zonoid, and let $\mathcal{F}=\{(1-t)K+tB:t\in[0,1]\}$ where
$B$ is the Euclidean unit ball. Then every $K'\in\mathcal{F}$ is a zonoid
of class $C^\infty_+$. Moreover, it was observed in Remark 
\ref{rem:protoequality} that every even eigenfunction of 
$\mathscr{A}_{K'}$ that is orthogonal to $h_{K'}$ has eigenvalue
$\lambda < -\frac{1}{n-1}$. By the continuity of the eigenvalues
of the Hilbert operator (cf. \cite[Theorem 5.3]{KM17}), there exists
$\varepsilon>0$ so that for every $K'\in\mathcal{F}$, every even 
eigenfunction of 
$\mathscr{A}_{K'}$ that is orthogonal to $h_{K'}$ has eigenvalue 
$\lambda \le -\frac{1}{n-1}-\varepsilon$. Thus there exists $p<0$ so that
condition (4) of \cite[Theorem 2.1]{Mil21} holds for 
all $K'\in\mathcal{F}$. The conclusion now follows from the
implication (4)$\Rightarrow$(3b) of \cite[Theorem 2.1]{Mil21} (as
the inequality in (3b) with $q=0$ is precisely \eqref{eq:logmink}).
\end{proof}

The logarithmic Brunn-Minkowski conjecture is intimately connected to
the uniqueness problem for cone volume measures; this was in fact the
original motivation for the formulation of the conjecture \cite{BLYZ12}.
Let us recall the definition.

\begin{defn}
The \emph{cone volume measure} $V_K$ of a convex body $K$ is
defined as
$$
	dV_K := \frac{1}{n}\,h_K\, dS_{K,\ldots,K}.
$$
\end{defn}

The basic question that arises here is whether the cone volume measure 
uniquely characterizes the convex body $K$. While this is not always the 
case, the question is closely connected to the equality cases of the 
logarithmic Minkowski inequality \eqref{eq:logmink} in the case that
$K,L\in\mathcal{K}_s^n$ are symmetric. For example, if 
$K,L\in\mathcal{K}_s^n$ satisfy $V_K=V_L$ (and thus \emph{a fortiori}
$\Vol(K)=\Vol(L)$ as $\Vol(K)=\int dV_K$), the validity of the logarithmic 
Brunn-Minkowski conjecture would yield
$$
	0\le \int h_K \log\bigg(\frac{h_L}{h_K}\bigg)\,dS_{K,\ldots,K}
	= \int h_L \log\bigg(\frac{h_L}{h_K}\bigg)\,dS_{L,\ldots,L}
	\le 0
$$
using \eqref{eq:logmink} in the first inequality, $V_K=V_L$ in the
equality, and \eqref{eq:logmink} with the roles of $K,L$ reversed in the 
second inequality. This would imply that $V_K$ is uniquely determined by 
$K$ whenever \eqref{eq:logmink} does not admit nontrivial equality 
cases.

Unfortunately, even though we obtained a complete characterization of the 
equality cases of \eqref{eq:locallogbm} when $K$ is a zonoid, this 
information is lost in Corollary \ref{cor:main}. The reason is that the 
proof of Corollary \ref{cor:main} required approximation of $K$ by smooth 
bodies, which destroys the nontrivial equality cases. Nonetheless, for 
sufficiently smooth zonoids, uniqueness of cone volume measures
follows by \cite[Theorem 2.1]{Mil21}. Note that while the smoothness 
assumption on $K$ is restrictive, the following statement requires neither 
that $L$ is smooth nor that $L$ is a zonoid.

\begin{cor}
Let $K\in\mathcal{K}^n_s$ be a zonoid of class $C^3_+$.
Then for any $L\in\mathcal{K}^n_s$, we have $V_K=V_L$ if and only if 
$K=L$.
\end{cor}

\begin{proof}
This follows from the
implication (4)$\Rightarrow$(1) of \cite[Theorem 2.1]{Mil21} by precisely 
the same argument as in the proof of Corollary \ref{cor:main}.
\end{proof}

\subsection*{Acknowledgments}

The results of this note were developed as a pedagogical example for the 
author's lectures at the Spring School on Convex Geometry and Random 
Matrices in High Dimension, Paris, June 2021. The author is grateful to 
M.\ Fradelizi, N.\ Gozlan, and O.\ Gu\'edon for the invitation to present 
these lectures. The author also thanks K.\ B\"or\"oczky, E.\ Milman, and 
G.\ Paouris for helpful discussions, and E.\ Milman and the anonymous 
referee for helpful suggestions on the presentation.
This work was supported in part by NSF grants DMS-1811735 
and DMS-2054565, and by the Simons Collaboration on Algorithms \& 
Geometry.

\bibliographystyle{abbrv}
\bibliography{ref}

\begin{thebibliography}{10}

\bibitem{Ale96}
A.~D. Alexandrov.
\newblock {\em Selected works. {P}art {I}}.
\newblock Gordon and Breach Publishers, Amsterdam, 1996.

\bibitem{Bol69}
E.~D. Bolker.
\newblock A class of convex bodies.
\newblock {\em Trans. Amer. Math. Soc.}, 145:323--345, 1969.

\bibitem{BK20}
K.~J. B\"{o}r\"{o}czky and P.~Kalantzopoulos.
\newblock Log-{B}runn-{M}inkowski inequality under symmetry, 2020.
\newblock Preprint arxiv:2002.12239.

\bibitem{BLYZ12}
K.~J. B\"{o}r\"{o}czky, E.~Lutwak, D.~Yang, and G.~Zhang.
\newblock The log-{B}runn-{M}inkowski inequality.
\newblock {\em Adv. Math.}, 231(3-4):1974--1997, 2012.

\bibitem{BL21}
P.~Br\"and\'en and J.~Leake.
\newblock Lorentzian polynomials on cones and the {H}eron-{R}ota-{W}elsh
  conjecture, 2021.
\newblock Preprint arxiv:2110.00487.

\bibitem{CP21}
S.~H. Chan and I.~Pak.
\newblock Log-concave poset inequalities, 2021.
\newblock Preprint arxiv:2110.10740.

\bibitem{CHLL20}
S.~Chen, Y.~Huang, Q.-R. Li, and J.~Liu.
\newblock The {$L^p$}-{B}runn-{M}inkowski inequality for $p<1$.
\newblock {\em Adv. Math.}, 368:107166, 2020.

\bibitem{CLM17}
A.~Colesanti, G.~V. Livshyts, and A.~Marsiglietti.
\newblock On the stability of {B}runn-{M}inkowski type inequalities.
\newblock {\em J. Funct. Anal.}, 273(3):1120--1139, 2017.

\bibitem{CKMS19}
D.~Cordero-Erausquin, B.~Klartag, Q.~Merigot, and F.~Santambrogio.
\newblock One more proof of the {A}lexandrov-{F}enchel inequality.
\newblock {\em C. R. Math. Acad. Sci. Paris}, 357(8):676--680, 2019.

\bibitem{Gar02}
R.~J. Gardner.
\newblock The {B}runn-{M}inkowski inequality.
\newblock {\em Bull. Amer. Math. Soc. (N.S.)}, 39(3):355--405, 2002.

\bibitem{GMR88}
Y.~Gordon, M.~Meyer, and S.~Reisner.
\newblock Zonoids with minimal volume-product---a new proof.
\newblock {\em Proc. Amer. Math. Soc.}, 104(1):273--276, 1988.

\bibitem{Hil12}
D.~Hilbert.
\newblock {\em Grundz\"{u}ge einer allgemeinen {T}heorie der linearen
  {I}ntegralgleichungen}.
\newblock B. G. Teubner, 1912.

\bibitem{KM18}
A.~V. Kolesnikov and E.~Milman.
\newblock Poincar\'{e} and {B}runn-{M}inkowski inequalities on the boundary of
  weighted {R}iemannian manifolds.
\newblock {\em Amer. J. Math.}, 140(5):1147--1185, 2018.

\bibitem{KM17}
A.~V. Kolesnikov and E.~Milman.
\newblock Local {$L^p$}-{B}runn-{M}inkowski inequalities for {$p<1$}.
\newblock {\em Mem. Amer. Math. Soc.}, 277(1360):v+78, 2022.

\bibitem{Mil21}
E.~Milman.
\newblock Centro-affine differential geometry and the log-{M}inkowski problem,
  2021.
\newblock Preprint arxiv:2104.12408.

\bibitem{Ore62}
O.~Ore.
\newblock {\em Theory of graphs}.
\newblock American Mathematical Society Colloquium Publications, Vol. XXXVIII.
  American Mathematical Society, Providence, R.I., 1962.

\bibitem{Put21}
E.~Putterman.
\newblock Equivalence of the local and global versions of the
  {$L^p$}-{B}runn-{M}inkowski inequality.
\newblock {\em J. Funct. Anal.}, 280(9):Paper No. 108956, 20, 2021.

\bibitem{Rei80}
R.~C. Reilly.
\newblock Geometric applications of the solvability of {N}eumann problems on a
  {R}iemannian manifold.
\newblock {\em Arch. Rational Mech. Anal.}, 75(1):23--29, 1980.

\bibitem{Rot14}
L.~Rotem.
\newblock A letter: The log-{B}runn-{M}inkowski inequality for complex bodies,
  2014.
\newblock Preprint arxiv:1412.5321.

\bibitem{Sar15}
C.~Saroglou.
\newblock Remarks on the conjectured log-{B}runn-{M}inkowski inequality.
\newblock {\em Geom. Dedicata}, 177:353--365, 2015.

\bibitem{Sch14}
R.~Schneider.
\newblock {\em Convex bodies: the {B}runn-{M}inkowski theory}.
\newblock Cambridge University Press, expanded edition, 2014.

\bibitem{SvH18}
Y.~Shenfeld and R.~van Handel.
\newblock Mixed volumes and the {B}ochner method.
\newblock {\em Proc. Amer. Math. Soc.}, 147(12):5385--5402, 2019.

\bibitem{SvH22}
Y.~Shenfeld and R.~{Van Handel}.
\newblock The extremals of {M}inkowski's quadratic inequality.
\newblock {\em Duke Math. J.}, 2022.
\newblock To appear.

\end{thebibliography}

\end{document}